\documentclass[10pt]{article}
\usepackage{graphicx,caption}
\usepackage{latexsym,booktabs}
\usepackage{blindtext,textcomp}
\usepackage{amsthm, amsmath, amssymb, amsbsy, xcolor, amsfonts}
\usepackage{hyperref}
\usepackage{tikz}
\usepackage{sectsty}
\sectionfont{\centering}
\usepackage{mathrsfs}
\usepackage{comment}
\usepackage{setspace} 
\usepackage{enumerate}
\usepackage[shortlabels]{enumitem}
\usepackage{lineno}
\usepackage[colorinlistoftodos]{todonotes}
\usepackage{adjustbox}
\usepackage{multicol}
\usepackage{comment}
\usepackage[utf8]{inputenc}
\usepackage[T1]{fontenc}

\newtheorem{theorem}{Theorem}[section]
\newtheorem{remark}[theorem]{Remark}
\newtheorem{cor}[theorem]{Corollary}

\newtheorem{prop}[theorem]{Proposition}
\newtheorem{lem}[theorem]{Lemma}

\def \N {{\mathbb N}}
\def \Z {{\mathbb Z}}
\def \R {{\mathbb R}}

\def \M {{\mathbb M}}

\def \la {{\lambda}}

\author{Priyanka Grover \thanks{Department of Mathematics, Shiv Nadar Institution of Eminence Delhi NCR, Gautam Buddha Nagar, Uttar Pradesh, 201314, India. \textbf{Email}: priyanka.grover@snu.edu.in} , Veer Singh Panwar \thanks{Department of Mathematics, Govt. Degree College Satpuli, Sri Dev Suman Uttarakhand University, Uttarakhand, 246172, India. \textbf{Email}: veerstar5@gmail.com}}
\date{}

\begin{document}
\title{\textbf{Inertia and other properties of the matrix $\left[\beta(i,j )\right]$}} 
\maketitle
 \begin{center}
\large{\textbf{Abstract}} 
 \end{center}
\begin{footnotesize}

\noindent Let $\pi(A)$, $\xi(A)$ and $\nu(A)$, respectively, denote the number of positive, zero and negative eigenvalues of the matrix $A$. Then the triplet $(\pi(A), \xi(A), \nu(A))$ is called the \emph{inertia} of $A$ and is denoted by $\textup{Inertia(A)}$. Let $\beta$ be the beta function. The inertia of the matrix $\left[\beta(i,j )\right]$ is shown to be $\left(\frac{n}{2},0,\frac{n}{2}\right)$ if $n$ is even, and $\left(\frac{n+1}{2},0,\frac{n-1}{2}\right)$ if $n$ is odd. 
It is also shown that $\left[\beta(i,j)\right]$ is Birkhoff-James orthogonal to the $n\times n$ identity matrix $I$ in the trace norm if and only if $n$ is even. 
For $0<\la_1<\cdots<\la_n, 0<\mu_1<\cdots<\mu_n$, it is shown that the matrix $\left[(\beta(\la_i,\mu_j))^m\right]$ is non singular if $\mu_{i+1}-\mu_{i}\in \N$ for all $1\leq i \leq n-1$. It is also shown that if $\mu_{i+1}-\mu_i \in \N$ for $1\leq i\leq n-1$, then for $m\in \mathbb N$, the matrix $\left[\frac{1}{\beta(\la_i,\mu_j)^m}\right]$ is totally positive. 
\end{footnotesize}

\vspace{0.5 cm}


\begin{footnotesize}
\end{footnotesize}

\vspace{1 cm}

\section{Introduction}
 Let $\beta(\cdot,\cdot)$ and $\Gamma(\cdot)$ denote the \emph{beta function} and  \emph{gamma function}, respectively. For $x,y \in \R^+$, $\beta(x,y)=\frac{\Gamma (x) \Gamma(y)}{\Gamma(x+y)}$.
 Let $\mathcal{B}=\left[\frac{1}{\beta(i,j)}\right]=\left[\frac{(i+j-1)!}{(i-1)!(j-1)!}\right]$ denote the \emph{beta matrix}~~\cite{grover2020positivity}. We shall assume that $1\leq i,j \leq n$, unless otherwise stated. 
 The matrix $\mathcal{B}$ is closely related to the \emph{Pascal matrix} $\mathcal{P}=\left[\frac{(i+j)!}{i!j!}\right]_{i,j=0}^{n-1}$.  For MATLAB toolbox for $\mathcal{B}$, see~\cite{higham2022anymatrix}. For numerical results and its usage in creating test matrices, see ~\cite{bujok2022computing, pena}. See \cite{karp} for results on ratios of fractional series related to the gamma kernels.
 
One of the main objectives of this paper is to study some interesting properties of a related matrix $\left[\beta(i,j)\right]$. This is the \emph{Hadamard} or \emph{Schur} inverse of 
$\mathcal{B}$, generally denoted as $\mathcal{B}^{\circ(-1)}$. Studying properties of Hadamard inverses of some well known matrices has found usage in~\cite{bhatia2016inertia,haukkanen2018inertia,prodinger2015reciprocal,richardson2014reciprocal}.

Let $\pi(A)$, $\xi(A)$ and $\nu(A)$, respectively, denote the number of positive, zero and negative eigenvalues of the matrix $A$. Then the triplet $(\pi(A), \xi(A), \nu(A))$ is called the \emph{inertia} of $A$ and is denoted by $\textup{Inertia(A)}$. Finding inertia of matrices is an important problem in matrix theory, see~\cite{bhatia2015inertia,bhatia2016inertia,dyn,haukkanen2018inertia,jain2017hadamard,jain2020hadamard}. Our main result is the inertia of the matrix $\left[\beta(i,j)\right]$.

\begin{theorem}\label{inertia of beta had powers}
We have 
\begin{align*} 
\textup{Inertia}\left(\left[\beta(i,j)\right]\right) = \begin{cases}
     \left(\dfrac{n}{2},0,\dfrac{n}{2}\right) \,\,  & \text{if} \,\, n \,\, \textup{is even},\\[0.5 cm]
     \left(\dfrac{n+1}{2},0,\dfrac{n-1}{2}\right) \,\, & \text{if} \,\, n \,\, \textup{is odd}.
    \end{cases}
   \end{align*}
\end{theorem}

 Let $\|\cdot\|_1$ denote the \emph{trace norm} on the space $\M_n(\R)$ of $n \times n$ real matrices. Let $X, Y \in \M_n(\R)$. Then $X$ is said to be \emph{Birkhoff-James orthogonal} to $Y$ in the trace norm if \begin{equation}\|X+t Y\|_1\geq \|X\|_1\text{ for all }t\in \mathbb R.\label{synn:new}\end{equation}
This has been studied in detail in \cite{Bhatia, 2017,  pereira, lischneider}. For more details on this topic, we refer the reader to \cite{2020} and the references therein.

Let $I$ denote the identity matrix of order $n$. Substituting $H=[\beta(i,j)]$  and $B=I$ in Theorem~2.3 of \cite{pereira}, we get that the matrix $[\beta(i,j)]$ is Birkhoff-James orthogonal to $I$ in the trace norm if and only if $\pi([\beta(i,j)])\leq \frac{n}{2}$ and $\nu([\beta(i,j)])\leq \frac{n}{2}$. So, Theorem~\ref{inertia of beta had powers} gives us the following result. 

\begin{theorem}\label{1.2}
The $n \times n$ matrix $\left[\beta(i,j)\right]$ is Birkhoff-James orthogonal to the matrix $I$ in the trace norm if and only if $n$ is even.
\end{theorem}

Similar results about the Hadamard inverse of the Pascal matrix are also pointed out. A matrix is called an \emph{integer matrix} if all its entries are integers. The matrix $\mathcal{P}^{\circ(-1)}$ is not an integer matrix, but it was proved in~\cite{richardson2014reciprocal} that its inverse is. The \emph{Hilbert matrix} given by $H = \left[\dfrac{1}{i+j-1}\right]$ is not an integer matrix but its inverse is an integer matrix, see~\cite{berman2002inverse}. Likewise, there exist some other well-known non integer matrices, whose inverses are integer matrices, see~\cite{kilicc2015variant, prodinger2015reciprocal, richardson2014super}. 
The next main result shows that the inverse of the matrix $[\beta(i,j)]$ is also an integer matrix. To do so, we compute its entries explicitly.
For any positive integers $r,k$ 
let $\binom{r}{k}$ denote the {binomial coefficient}. 
The coefficient $\binom{0}{k}$ is $1$ if $k=0$, and $0$ otherwise. 

\begin{theorem}\label{inverse of had inv of beta}
The inverse of $\left[{\beta(i,j)}\right]$ is an integer matrix with $(i,j)$th entry given by 
\begin{align*}
    (-1)^{n+i-j}\binom{n+i-1}{i-1}\binom{n}{j}j\sum_{1\leq k\leq \textup{min}(i,j)}\binom{n-k}{n-i}\binom{n+j-1}{n+k-1}(-1)^k.
\end{align*}
\end{theorem}


Section~\ref{Section 2} is devoted to the proofs of the above results (Theorem~\ref{inertia of beta had powers} and Theorem~\ref{inverse of had inv of beta}). In Section \ref{section 2}, we consider the more general matrix $\left[{\beta(\la_i,\mu_j)}\right],$ where $0< \la_1<\cdots<\la_n$ and $0< \mu_1<\cdots<\mu_n$. For $m\in \mathbb N$, the matrix $ \left[{\beta(\la_i,\mu_j)}^m\right]$ denotes its $m$th \emph{Hadamard power}. One of the main results of this section is that the matrix $ \left[{\beta(\la_i,\mu_j)}^m\right]$ is nonsingular if $\mu_{i+1}-\mu_i\in \N$ for all $1\leq i \leq n-1$. 
In Section~\ref{remarks are here}, some remarks are given.

\section{Proofs}\label{Section 2}

To prove Theorem~\ref{inertia of beta had powers}, we first compute the determinant of $\left[{\beta(i,j)}\right].$

\begin{prop}\label{determinant}
    The determinant of the $n\times n$ matrix $\left[{\beta(i,j)}\right]$ is given by 
    \begin{equation}
\det \left(\left[{\beta(i,j)}\right]\right)
=(-1)^{n(3n+1)/2}\prod_{i=1}^{n}\frac{1}{\binom{n+i-1}{n}\binom{n}{i}i}.\label{xx}
\end{equation} 
\end{prop}

\begin{proof}
    Let $K = \left[\dfrac{1}{(i+j-1)!}\right]$. Then \begin{align}\label{eqn for schur inverse of beta}
\left[{\beta(i,j)}\right]=\text{diag}[(i-1)!]\cdot K \cdot\text{diag}[(i-1)!].
\end{align}
Let $A$ and $B$ be the matrices whose $(i,j)$th entries are given by 
\begin{align*}
A_{ij}&= \begin{cases}
    \displaystyle \binom{n-j}{n-i}(-1)^j &\text{if} \,\, i\geq j,\\[0.5 cm]
    0 & \text{otherwise},
    \end{cases}
\intertext{and} 
    B_{ij}&= \begin{cases}(-1)^{i-j}\displaystyle \binom{n+j-1}{n+i-1}\,\,& \text{if}\,\, i\leq j,\\[0.5 cm]0 &\text{otherwise},\end{cases}
\end{align*}
respectively. Let $D_{1} = \text{diag}\left[\frac{(-1)^{n-i}} {(n+i-1)!}\right]$ and $D_{2} = \text{diag}\left[(-1)^i(n-i)!\right]$.
We first show that \begin{align}{\label{new equation for K}}
K = D_2BAD_1.
\end{align}
To prove this, we show that 
$BA = D_2^{-1}KD_1^{-1}$  that is,
\begin{align} \label{eqn: result}
\sum_{k= \textup{max}(i,j)}^{n}\binom{n+k-1}{n+i-1}\binom{n-j}{n-k}(-1)^{i-k+j}&=\frac{(-1)^{n+j-i}(n+j-1)!}{(n-i)!(i+j-1)!}.
\end{align}
Let $\rho, \tau \in \R$ and $a,b,c \in \Z$. Consider the following combinatorial identities given in \cite[p. 174]{graham1989concrete} and \cite[p. 169]{graham1989concrete}, respectively:
 \begin{align} \label{eqn: upper negation}
    \binom{\tau}{c}&=(-1)^{c}\binom{c-\tau-1}{c},
\end{align}
and
\begin{align} \label{eqn: concrete math}
    \sum_{c \in \Z}\binom{\tau}{a+c}\binom{\rho}{b-c} &= \binom{\tau+\rho}{a+b}. 
\end{align}
So,
\begin{align}
\sum_{k= \textup{max}(i,j)}^{n}\binom{n+k-1}{n+i-1}\binom{n-j}{n-k}(-1)^{k}&= \sum_{k= \textup{max}(i,j)}^{n}\binom{n+k-1}{k-i}\binom{n-j}{n-k}(-1)^{k}\nonumber\\
&=(-1)^{-i}\sum_{k= \textup{max}(i,j)}^{n}\binom{-n-i}{k-i}\binom{n-j}{n-k}\label{*}\\
&=(-1)^{-i}\binom{-i-j}{n-i}\label{**}\\
&=(-1)^{-i}(-1)^{n-i}\binom{n+j-1}{n-i}\label{***}\\
&=(-1)^{n} \binom{n+j-1}{n-i}\nonumber\\
&=(-1)^{n}\frac{(n+j-1)!}{(n-i)!(i+j-1)!}.\nonumber
\end{align}

 Here, \eqref{*} and \eqref{***} follow by \eqref{eqn: upper negation}, and \eqref{**} follows by \eqref{eqn: concrete math}. This shows \eqref{new equation for K}.
 
So, by~\eqref{eqn for schur inverse of beta} and~\eqref{new equation for K}, we get
\begin{align*}
\det \left(\left[{\beta(i,j)}\right]\right)&=\prod_{i=1}^{n}(-1)^n(-1)^i(n-i)!(i-1)!\frac{(i-1)!}{(n+i-1)!} \nonumber \\
&=(-1)^{n^2}(-1)^{n(n+1)/2}\prod_{i=1}^{n}(n-i)!(i-1)!\frac{(i-1)!}{(n+i-1)!}\\
&=(-1)^{n(3n+1)/2}\prod_{i=1}^{n}\frac{1}{\dfrac{(n+i-1)!}{n!(i-1)!}\dfrac{n!}{i!(n-i)!}(i)}.
\end{align*}
This proves \eqref{xx}.
\end{proof}

 For the next result, the $n\times n$ matrix $\left[{\beta(i,j)}\right]$ is denoted by $\mathcal B_n^{\circ (-1)}$.
\begin{cor}\label{determinant sign}
 The determinants of $\mathcal{B}_n^{\circ (-1)}$ and $\mathcal{B}_{n+1}^{\circ (-1)}$ have same sign if $n$ is even, and opposite signs if $n$ is odd. 
\end{cor}
\begin{proof}
 From~\eqref{xx}, the sign of the product $\det(\mathcal{B}_n^{\circ (-1)})\det(\mathcal{B}_{n+1}^{\circ (-1)})$ is $(-1)^{\frac{n(3n+1)}{2}}\cdot(-1)^{\frac{(n+1)(3n+4)}{2}}$. This is same as $(-1)^{n(3n+4)}$, which is positive if $n$ is even, and negative, otherwise.
\end{proof}

\textit{Proof of Theorem \ref{inertia of beta had powers}}  The result is obvious for $n=1$. Let the result hold for any odd integer $n$. Then $\textup{Inertia}\left(\mathcal{B}_n^{\circ (-1)}\right) = \left(\frac{n+1}{2},0,\frac{n-1}{2}\right).$ Therefore, by Cauchy's interlacing theorem~\cite[Theorem 4.3.17]{horn2012matrix}, the matrix $\mathcal{B}_{n+1}^{\circ (-1)}$ has at least $\frac{n+1}{2}$ positive, and $\frac{n-1}{2}$ negative eigenvalues. By Corollary \ref{determinant sign}, the remaining eigenvalue of $\mathcal{B}_{n+1}^{\circ (-1)}$ will be negative. Thus $\textup{Inertia}\left(\mathcal{B}_{n+1}^{\circ (-1)}\right)= \left(\frac{n+1}{2},0,\frac{n+1}{2}\right).$ Therefore, the result holds for $n+1$. A similar argument shows that if the result holds for any even positive integer $n$, then it holds for $n+1$. \qed 
\vspace{0.2cm}

\textit{Proof of Theorem \ref{inverse of had inv of beta}}
    From \eqref{eqn for schur inverse of beta}, $$\left[{\beta(i,j)}\right]=\text{diag}[(i-1)!]\cdot K \cdot\text{diag}[(i-1)!].$$
This gives 
\begin{align}\label{decomp}
\left[{\beta(i,j)}\right]^{-1} = \text{diag}\left[\frac{1}{(i-1)!}\right] \cdot K^{-1}\cdot\text{diag}\left[\frac{1}{(i-1)!}\right].
\end{align}
Since $\displaystyle \binom{n+i-1}{i-1}\binom{n}{j}j= \dfrac{(n+i-1)!}{(i-1)!(n-j)!(j-1)!}$, it is enough to show that 
\begin{align}\label{eqn: eqn for inverse of K}
    K^{-1} = \left[(-1)^{n+i-j}\frac{(n+i-1)!}{(n-j)!}\sum_{1\leq k\leq \textup{min}(i,j)}\binom{n-k}{n-i}\binom{n+j-1}{n+k-1}(-1)^k\right].
\end{align}
From \eqref{new equation for K}, we have 
\begin{align}
K^{-1}=D_1^{-1}A^{-1}B^{-1}D_2^{-1}.\label{new equation}
\end{align}
The $(i,j)$th entry of $A^2$ is given by
\vspace{0.2cm}
\begin{align*}
    (A^{2})_{ij} &= \sum_{j\leq k\leq i}\binom{n-k}{n-i}\binom{n-j}{n-k}(-1)^{k+j}.
\end{align*}   
Substituting $n-k = m$, we get
\begin{align*}
(A^{2})_{ij}&= (-1)^{j+n}\sum_{m=n-i}^{n-j}\binom{m}{n-i}\binom{n-j}{m}(-1)^{-m}\\
&=(-1)^{j+n}\sum_{m=n-i}^{n-j}\binom{m}{n-i}\binom{n-j}{m}(-1)^{m}.
\end{align*}
By \cite[eq. 5.24]{graham1989concrete},
$$(A^{2})_{ij}=(-1)^{j+n}(-1)^{n-j}\binom{0}{j-i}=\delta_{ij}.$$
This gives that $A^{-1} = A$. Next, we show that \begin{equation}(B^{-1})_{ij} = \begin{cases}\displaystyle \binom{n+j-1}{n+i-1}\,\,& \text{if}\,\, i\leq j,\\[0.5 cm]0 &\text{otherwise}.\end{cases}\label{new}\end{equation}
For this, consider the  sum 
$$\sum_{i\leq k\leq j}\binom{n+k-1}{n+i-1}(-1)^{k-j}\binom{n+j-1}{n+k-1}.$$ Putting $n+k-1 = m$, this sum becomes
\begin{align*}
&\quad \ (-1)^j\sum_{m = n+i-1}^{n+j-1}\binom{m}{n+i-1}\binom{n+j-1}{m}(-1)^{m+1-n}\\
&=(-1)^j(-1)^{1-n}\sum_{m = n+i-1}^{n+j-1}\binom{m}{n+i-1}\binom{n+j-1}{m}(-1)^{m}.
\end{align*}
By~\cite[eq. 5.24]{graham1989concrete}, this is same as
\begin{align*}
    (-1)^j(-1)^{1-n}(-1)^{n+j-1}\binom{0}{i-j} = \binom{0}{i-j} = \delta_{ij}. 
\end{align*}
This shows \eqref{new}. 
Now, from~\eqref{new equation}, 
\begin{align*}
    (K^{-1})_{ij} &= (D_1^{-1})_{ii}(A^{-1}B^{-1})_{ij}(D_2^{-1})_{jj}\\
    &=(-1)^{n-i}(n+i-1)!\left(\sum_{1\leq k \leq \text{min}(i,j)}(-1)^{k}\binom{n-k}{n-i}\binom{n+j-1}{n+k-1}\right)\frac{(-1)^j}{(n-j)!}.
\end{align*}
This proves \eqref{eqn: eqn for inverse of K}. \qed
\section{Some properties of $\left[\beta(\la_i,\mu_j)^{m}\right]$}\label{section 2}

It is clear from Proposition \ref{determinant} that $\left[\beta(i,j)\right]$ is non singular. The below theorem gives a more general result.

\begin{theorem} \label{res beta had pow is non singular}
Let $m \in \N$. Suppose $0< \la_1<\cdots<\la_n$ and $0< \mu_1<\cdots<\mu_n$. Let 
$\mu_{i+1}-\mu_i\in \N$ for all $1\leq i \leq n-1$. 
Then $\left[\frac{1}{(\Gamma(\lambda_{i}+ \mu_j))^{m}}\right]$ and $\left[\beta(\la_i,\mu_j)^{m}\right]$ are non singular.
\end{theorem}

To prove this result, we shall need the \emph{Descartes' rule of signs} \cite[p. 41, Ex. 36]{polya1997problems}: 
Let 
\begin{align}\label{poly}
p(x) = a_{p}x^p+a_{p-1}x^{p-1}+\cdots +a_{0}
\end{align}
be a real polynomial, where $a_p\neq 0$. Let $Z(p(x))$ denote the number of positive zeros of $p(x)$. Let $N(p(x))$ (or $N(a_p,a_{p-1},\ldots,a_0)$) denote the number of changes of signs in $a_p,a_{p-1},\ldots,a_0$ .  Then \begin{equation}Z(p(x)) \leq N(p(x)).\label{drs}\end{equation}


We now prove two lemmas.
\begin{lem}\label{not more than}
    Let $p(x)$ be a real polynomial of degree $n$ given in~\eqref{poly}. Let $\alpha >0$. Then we have 
    \begin{align}\label{equ: no. of zeros}
     N(p(x)(x+\alpha)) \leq N(p(x)).
    \end{align}
\end{lem}    

\begin{proof}
Expanding $p(x)(x+\alpha)$ in decreasing order of powers of the variable $x$, it has coefficients $(a_{p},a_{p-1}+a_{p}\alpha,\ldots, a_{0}+a_{1}\alpha, a_{0}\alpha)$. 
Since $\alpha>0$, we have
\begin{align*}
    N(a_{p},a_{p-1}+ & a_{p}\alpha,\ldots,a_1+a_2\alpha,a_{0}+a_{1}\alpha,a_{0}\alpha) \\
    &= N(a_{p}\alpha^p,(a_{p-1}+a_{p}\alpha)\alpha^{p-1},\ldots,(a_1+a_2\alpha)\alpha,(a_{0}+a_{1}\alpha), a_{0})\\
    &=N(a_{p}\alpha^p,a_{p-1}\alpha^{p-1}+a_{p}\alpha^p,\ldots,a_1\alpha+a_2\alpha^2,a_{0}+a_{1}\alpha, a_{0}).
\end{align*}
Now by \cite[Ex. 4]{polya1997problems}
, we get
\begin{align*}
N(a_{p}\alpha^p,a_{p-1}\alpha^{p-1}+a_{p}\alpha^p,\ldots,a_1\alpha+a_2\alpha^2,a_{0}+a_{1}\alpha, a_{0})&\leq N(a_p\alpha^p,a_{p-1}\alpha^{p-1},\ldots,a_2\alpha^2,a_1\alpha,a_0).
\end{align*}
Again, since $\alpha>0$,
the right hand side is equal to 
\begin{align*}
N(a_p, a_{p-1},\ldots, a_2, a_1, a_0).
\end{align*}
\noindent Hence, we have 
    \begin{align*} 
     N(p(x)(x+\alpha)) \leq N(p(x)).
    \end{align*} 
\end{proof}

\begin{lem} \label{Lemma: at most $n-1$ zeros}
Let $m$ and $p$ be fixed positive integers. Let $l_1, \ldots, l_p$ be positive integers,  and $c_1,\ldots, c_{p+1}$ be arbitrary real numbers. Suppose $\alpha_{kt}>0$ for all $1 \leq k\leq p, 1 \leq t\leq \textup{max}\{l_1,\ldots,l_p \}$. Let \begin{align}\label{eq: recurring function}
f_{1}(x) &= c_{1}\prod_{t=1}^{l_1}(x+\alpha_{1t})^{m}+c_{2}\nonumber\intertext{and for $2 \leq k \leq p$, let}
f_{k}(x) &= f_{k-1}(x)\prod_{t=1}^{l_{k}}(x+\alpha_{kt})^{m}+c_{k+1}.
\end{align} 
 If $c_1,\ldots,c_{p+1}$ are not all zero, then the polynomial $f_{p}(x)$ has at most $p$ positive zeros.
\end{lem}

\begin{proof}
Suppose $c_1,\ldots,c_{p+1}$ are not all zero. Let $q$ be the least index such that $c_q \neq 0$. Then the leading coefficient of $f_p(x)$ is $c_q$. So $f_p(x)$ can not be identically zero on $\R^+$. 

\par 

Next we show that if the polynomial $f_k(x)$ is written in decreasing order of powers of the  variable $x$, then $N(f_{k}(x))\leq k$ for every $1 \leq k \leq p$. We will prove this by applying induction on $k$. Since $\prod_{t=1}^{l_1}(x+\alpha_{1t})^{m}$ is a polynomial with positive coefficients, the result holds easily for the case when $k=1$. Let us assume that the result holds for $k=v-1\geq 1$. Since adding a constant term to a polynomial can increase the number of changes of signs by at most one, we have
\begin{align*}
          N(f_{v}(x)) &= N\left(f_{v-1}(x)\prod_{t=1}^{l_v}(x+\alpha_{vt})^{m}+c_{v+1}\right)\\
          &\leq N\left(f_{v-1}(x)\prod_{t=1}^{l_v}(x+\alpha_{vt})^{m}\right)+1.
\end{align*}
So by Lemma~\ref{not more than}, we get
\begin{align*}
N(f_{v}(x))\leq N(f_{v-1}(x))+1
\leq (v-1)+1=v.
\end{align*}
Hence, $N(f_{p}(x)) \leq p$. So by \eqref{drs}, we have $Z(f_{p}(x)) \leq N(f_{p}(x)) \leq p$, as desired.
\end{proof}

\hspace{-0.55cm}\textit{Proof of Theorem~\ref{res beta had pow is non singular}}.
Suppose for some $\la_i, \mu_j, 1\leq i,j \leq n$ and $m \in \N$, the matrix $\left[\frac{1}{(\Gamma(\lambda_{i}+ \mu_j))^{m}}\right]$ is singular. Then there exists a nonzero vector $c = (c_{1}, \ldots, c_{n}) \in \R^{n}$ such that  
\begin{align}\label{eqs: zeros equation}
\sum_{j=1}^{n}\frac{c_{j}}{{(\Gamma(\lambda_{i}+\mu_j))^{m}}} & = 0 \quad \text{for} \,\, i = 1, \ldots, n.
\end{align}
\noindent Let 
$f(x)= \sum_{j=1}^{n}\frac{c_{j}}{{(\Gamma(x+\mu_j))^{m}}}$. Taking $(\Gamma(x+\mu_n))^m$ as the LCM of the denominators, and using the property that $\Gamma(x+1) = x\Gamma(x)$ for all $x\in \R^+$, we get
\begin{align*}
f(x)=\frac{1}{(\Gamma(x+\mu_n))^{m}}\left[\left\{\sum_{j=1}^{n-1} c_{j} \prod_{k=0}^{\mu_n-\mu_j-1}(x+\mu_j+k)^m \right\} +c_n\right].
\end{align*}
\noindent Let
\begin{align*}
h_{1}(x) &= c_{1}\prod_{l=\mu_1}^{\mu_2-1}(x+l)^{m}+c_{2}\nonumber\intertext{and}
h_{k}(x) &= h_{k-1}(x)\prod_{l=\mu_k}^{\mu_{k+1}-1}(x+l)^{m}+c_{k+1}\,\,\textup{for}\,\, 2 \leq k \leq n-1.
\end{align*}
Then $f(x) = \frac{1}{(\Gamma(x+\mu_n))^{m}} h_{n-1}(x)$.
By Lemma~\ref{Lemma: at most $n-1$ zeros}, $h_{n-1}(x)$ has at most $(n-1)$ positive zeros. Hence, the function $f(x)$ can have at most $(n-1)$ positive zeros, which is a contradiction to \eqref{eqs: zeros equation}. Thus $\left[\frac{1}{(\Gamma(\lambda_{i}+ \mu_j))^{m}}\right]$ is non singular. Since $\beta(\la_i,\mu_j)=\frac{\Gamma (\la_i) \Gamma(\mu_j)}{\Gamma(\la_i+\mu_j)}$, we have $$\left[\beta(\la_i,\mu_j)^m\right]=\text{diag}\left[(\Gamma(\la_i))^m\right]\left[\frac{1}{(\Gamma(\la_i+\mu_j))^m}\right]\text{diag}\left[(\Gamma(\mu_i))^m\right].$$ Hence,  $\left[\beta(\la_i,\mu_j)^m\right]$ is non singular.

\qed

A matrix $A$ is called \emph{totally nonnegative} (or \emph{totally positive}) if all its minors are nonnegative (or positive)~\cite{fallat2011totally}. For $r>0$, the total positivity of $\left[\frac{1}{\beta(i,j)^r}\right]$ was shown in~\cite{grover2020positivity}. Some results for its bidiagonal decomposition were discussed in~\cite{grover2022bidiagonal}.
In \cite{grover2020positivity}, it was shown that if $0<\la_1<\cdots<\la_n$ and $0<\mu_1<\cdots<\mu_n$, then $\left[\frac{1}{\beta(\la_i,\mu_j)}\right]$  is totally positive. Further, in \cite{pena}, bidiagonal decompositions of the matrix $\left[\frac{1}{\beta(\la_i,\mu_j)}\right]$ were computed for all  $\mu_{i+1}-\mu_i=1$, $1\leq i\leq n-1$,. The below result is for total positivity of Hadamard powers of $\left[\frac{1}{\beta(\la_i,\mu_j)}\right]$.
\begin{theorem}
Let $\mu_{i+1}-\mu_i \in \N$ for $1\leq i\leq n-1$. For $m\in \mathbb N$, the matrices $[(\Gamma(\la_i+\mu_j))^m]$ and $\left[\frac{1}{\beta(\la_i,\mu_j)^m}\right]$ are totally positive.
\end{theorem} 
\begin{proof}
By Lemma $4.8$ of \cite{khare2020multiply}, the matrix $[f(\Gamma(\la_i+\mu_j))]$ is totally nonnegative for every power series $f$ with nonnegative coefficients.
In particular, for $m\in \mathbb N$, the matrix $[(\Gamma(\la_i+\mu_j))^m]$ is totally nonnegative. 
We have 
\begin{align}\label{beta gamma relation}
\left[\frac{1}{\beta(\la_i,\mu_j)^m}\right]=\text{diag}\left[\frac{1}{(\Gamma(\la_i))^m}\right]\left[(\Gamma(\la_i+\mu_j))^m\right]\text{diag}\left[\frac{1}{(\Gamma(\mu_i))^m}\right].
\end{align}
The product of a totally nonnegative (or totally positive) matrix with a diagonal matrix with positive diagonal entries is also totally nonnegative (or totally positive) (see~\cite[p. 34]{fallat2011totally}). So, by~\eqref{beta gamma relation}, the matrix $\left[\frac{1}{\beta(\la_i,\mu_j)^m}\right]$ is also totally nonnegative. 
Further, let $\mu_{i+1}-\mu_i \in \N$ for $1\leq i\leq n-1$. By the same approach as used in the proof of Theorem~\ref{res beta had pow is non singular}, we can prove that the matrix $[(\Gamma(\la_i+\mu_j))^m]$ 
is non singular. Every submatrix of $[(\Gamma(\la_i+\mu_j))^m]$ is of the same form. Thus, we have that every minor of $[(\Gamma(\la_i+\mu_j))^m]$ is positive, that is, the matrix $[(\Gamma(\la_i+\mu_j))^m]$ is totally positive. Hence, by~\eqref{beta gamma relation}, the matrix $\left[\frac{1}{\beta(\la_i,\mu_j)^m}\right]$ is also totally positive.
\end{proof}

\section{Remarks}\label{remarks are here}

\begin{remark}
From~\eqref{decomp} and~\eqref{new equation}, the inverse of the matrix $\left[{\beta(i,j)}\right]$ has an LU decomposition where the $(i,j)$th entries of $L$ and $U$ are given by
   \begin{align*}
    L_{ij}&= \begin{cases}
    n!\binom{n-j}{n-i}\binom{n+i-1}{i-1}(-1)^{n+i+j} \quad &\text{if}\,\, i\geq j\\
    0 &\text{otherwise},
    \end{cases}\intertext{and}
    U_{ij}&= \begin{cases}
    \frac{1}{n!}\binom{n+j-1}{n+i-1}\binom{n}{j}j(-1)^{j} \quad &\text{if} \,\, i \leq j\\
    0 &\text{otherwise},
    \end{cases}
\end{align*} 
respectively.
\end{remark}

\begin{remark}
Let $0\leq i,j \leq n-1$.  Let $\mathcal{L} = \left[\binom{2i}{i+j}\right]$ be the lower triangular matrix with diagonal entries $1$, and let $G = \emph{diag}\left[\binom{2i}{i}\right]$. Let $D$ be the diagonal matrix, where $D_{0,0} =1$, and $D_{ii}=(-1)^{i}2$ for $0<i\leq n-1$. By Equation $(9)$ of \cite{richardson2014reciprocal}, $\mathcal{P}^{\circ(-1)} = G^{-1}\mathcal{L}D\mathcal{L}^TG^{-1}$. Since, $G$ and $\mathcal{L}$ have positive determinants, the sign of $\det \left(\mathcal{P}^{\circ (-1)}\right)$ is the same as that of  $\det D$, which is  
$(-1)^{n(n-1)/2}$. In Theorem~3.1 of~\cite{richardson2014reciprocal}, this was incorrectly calculated as $(-1)^{n(n+1)/2}$. Let $\mathcal{P}_n^{\circ (-1)}$ denote the Hadamard inverse of the Pascal matrix of order $n$. Then the sign of the product $(\det(\mathcal{P}_n^{\circ (-1)})\cdot\det(\mathcal{P}_{n+1}^{\circ (-1)}))$ is $(-1)^{n^2}$, which is positive if $n$ is even, and negative if $n$ is odd. Now, using a similar argument as in the proof of Theorem \ref{inertia of beta had powers}, we get
\begin{eqnarray} 
\textup{Inertia}\left(\mathcal{\mathcal{P}}^{\circ (-1)}\right) = \begin{cases}
     \left(\dfrac{n}{2},0,\dfrac{n}{2}\right) \,\,  & \text{if} \,\, n \,\, \textup{is even},\\[0.5 cm]
     \left(\dfrac{n+1}{2},0,\dfrac{n-1}{2}\right) \,\, & \text{if} \,\, n \,\, \textup{is odd}.
    \end{cases}\label{pascalinertia}
   \end{eqnarray}
   Thus, the $n \times n$ Pascal matrix is Birkhoff-James orthogonal to the $n\times n$ identity matrix $I$ in the trace norm if and only if $n$ is even.
\end{remark}

\textbf{Acknowledgement}. It is a pleasure to record our sincere thanks to the  anonymous referees for their thorough review and valuable remarks.

\end{document}